\documentclass[12pt,reqno]{amsart}
\usepackage{amsfonts}
\usepackage{amssymb, amsmath, latexsym}
\usepackage{lipsum} % for filler text
\usepackage{mathtools}
\DeclarePairedDelimiter\floor{\lfloor}{\rfloor}
\usepackage{array}
\usepackage[bottom]{footmisc}

\newcolumntype{M}[1]{>{\centering\arraybackslash}m{#1}}
\setbox0=\hbox{$+$}
\newdimen\plusheight
\plusheight=\ht0
\def\+{\;\lower\plusheight\hbox{$+$}\;}

\setbox0=\hbox{$-$}
\newdimen\minusheight
\minusheight=\ht0
\def\-{\;\lower\minusheight\hbox{$-$}\;}

\setbox0=\hbox{$\cdots$}
\newdimen\cdotsheight
\cdotsheight=\plusheight%\ht0
\def\cds{\lower\cdotsheight\hbox{$\cdots$}}

\newcommand\cplus{\mathbin{\raisebox{-\height}{$+$}}}
\newcommand\contdots{\raisebox{-\height}{$\vphantom{+}\dotsm$}}

\usepackage{mleftright}
\newcommand{\Leg}[3][]{\mleft(\frac{#2\mathstrut}{#3}\mright)_{\mkern-6mu#1}}
\renewcommand{\(}{\left\(}
\renewcommand{\)}{\right\)}
\renewcommand{\[}{\left[}

\numberwithin{equation}{section}
 \theoremstyle{plain}
\newtheorem{theorem}{Theorem}[section]
\newtheorem{lemma}[theorem]{Lemma}

\newtheorem{remark}[theorem]{Remark}

\newcommand\T{\rule{0pt}{2.6ex}}
\newcommand\B{\rule[-1.2ex]{0pt}{0pt}}
\makeatletter
\def\@bignumber#1#2{%
  \ifx#2\end
    #1\let\next\@gobble
  \else
    #1\hspace{0pt plus 1pt}\let\next\@bignumber
  \fi
  \next#2}
\newcommand{\bignumber}[1]{\@bignumber#1\end}
\makeatother

\begin{document}
\allowdisplaybreaks
\title[Families of congruences for fractional partition functions modulo powers of primes] {Families of congruences for fractional partition functions modulo powers of primes}

\author{Nayandeep Deka Baruah}
\address{Department of Mathematical Sciences, Tezpur University, Assam, India, Pin-784028}
\email{nayan@tezu.ernet.in}
\author{Hirakjyoti Das}
\address{Department of Mathematical Sciences, Tezpur University, Assam, India, Pin-784028}
\email{hdas@tezu.ernet.in}

%\maketitle

\begin{center}
{\textbf{Families of Congruences for Fractional Partition Functions Modulo Powers of Primes}}\\[5mm]
{\footnotesize  Nayandeep Deka Baruah and Hirakjyoti Das}\\[3mm]
\end{center}

\vskip 5mm \noindent{\footnotesize{\bf Abstract.} Recently, Chan and Wang (Fractional powers of the generating function for the partition function. Acta Arith. \textbf{187(1)},  59--80 (2019)) studied the fractional powers of the generating function for the partition function and found several  congruences satisfied by the corresponding coefficients. In this paper, we find some new families of congruences modulo powers of primes. We also find analogous results for the coefficients of the fractional powers of the generating function for the 2-color partition function.


\vskip 3mm
\noindent {\footnotesize 2010 Mathematical Reviews Classification
Numbers: Primary 05A17; Secondary 11P83}.
}

\section{\bf Introduction}\label{Introduction}

For complex numbers $a$ and $q$ with $|q|<1$, define the standard $q$-product by
\begin{align*}
       (a;q)_{\infty}&:=\prod_{j=0}^\infty(1-a q^j).
\end{align*}
In the sequel, for brevity, we set $E_n:=(q^n;q^n)_{\infty}$ for integers $n\ge1$.

A partition $\lambda=(\lambda_1, \lambda_2,\ldots,\lambda_k)$ of a positive integer $n$ is a finite non-increasing sequence of positive integers $\lambda_1, \lambda_2,\ldots,\lambda_k$ such that
$$\sum_{j=1}^k\lambda_j=n.$$
We call $\lambda_j$'s as parts of $\lambda$. The partition function $p(n)$ is defined as the number of partitions of $n$. It is well known  that the generating function of $p(n)$ is given by
\begin{align*}
    \sum_{n=0}^\infty p(n)q^n=\dfrac{1}{E_1},
   \end{align*}
   where by convention $p(0)=1$. Arithmetical properties of the partition function $p(n)$ have been studied quite extensively after Ramanujan
\cite{ram1}--\cite{ram3} found his famous congruences modulo 5, 7, and 11, namely, for all $n\ge 0$,
\begin{align*}
p(5n+4)&\equiv 0~(\textup{mod}~5),\\
p(7n+5)&\equiv 0~(\textup{mod}~7),\\\intertext{and}
p(11n+6)&\equiv 0~(\textup{mod}~11).\notag
\end{align*}

Now, for any non-zero rational number $t$, define  $p_t(n)$ by $$\sum_{n=0}^\infty p_t(n)q^n=E_1^t.$$
Clearly, $p_{-1}(n)=p(n)$. Ramanujan \cite[p. 182]{lnb} also initiated the study of the function $p_t(n)$ for non-zero integer $t\ne1$.  We refer to the paper  by Berndt, Gugg and Kim \cite{bgk} for further comments on Ramanujan's study. For some more work on $p_t(n)$ for non-zero integer $t\ne1$, see \cite{Atkin, baruah-sarmah, Boylan, Farkas, Gordon, Kiming-olsson,  Locus-Wagner,  Newman}. Note that $p_{-t}(n)$ for positive integer $t$ counts the number of $t$-colored partitions of $n$ where in a $t$-colored partition each part can have at most $t$ colors. 

In \cite{S.T.Ng}, by using the theory of modular forms, Ng found that, for all $n\geq0$, $$p_{-2/3}(19n+9)\equiv0~(\textup{mod}~19),$$ where the interpretation of this type of congruence is explained in the next paragraph.

Recently, Chan and Wang \cite{Chan-Wang} systematically initiated the study of the function $p_{t}(n)$ for non-integral rational $t$  and found numerous congruences for the function. In fact, they established the following important theorem which makes sense to study congruences of $p_{t}(n)$ modulo any positive integer $m$ such that
gcd$(m, \textup{denom}(t)) = 1$, where $\textup{denom}(t)$ denotes the denominator of $t$ for a rational number $t$ in the reduced form.
\begin{theorem}\label{CW-main1}\textup{(Chan-Wang \cite[Theorem 1.1]{Chan-Wang})} For any integer $n$ and prime $\ell$, let $\textup{ord}_{\ell}(n)$ denote the integer $k$ such that $\ell^k\mid n$ and $\ell^{k+1}\nmid n$.
Let $t=a/b$, where $a,b \in \mathbb{Z}$, $b\geq 1$ and $\textup{gcd}(a,b)=1$. We have
\begin{align*}
    \textup{denom}(p_t(n))=b^n\prod_{\ell \mid b}\ell^{\alpha_{\ell}(n)},
\end{align*}
where
\begin{align*}
    \alpha_{\ell}(n)=\textup{ord}_{\ell}(n!)=\floor*{\dfrac{n}{\ell}}+\floor*{\dfrac{n}{\ell^2}}+\cdots.
\end{align*}
\end{theorem}
\noindent From Theorem \ref{CW-main1}, we see that it is worthwhile to explore congruences for $p_t(n)$ modulo powers of prime $\ell$ such that $\ell \nmid \textup{denom}(t)$. Note that, for positive integers $A$ and $B$, a congruence of the form $p_{t}(An+B)\equiv0~(\textup{mod}~\ell^k)$, for all $n\ge 0$, should be interpreted as the numerator of $p_t(An+B)$ is divisible by $\ell^k$. By using the known series expansion of $E_1^d$ for
$d\in \{1, 3, 4, 6, 8, 10, 14, 26\}$, they proved the following theorem which gives Ramanujan-type congruences for $p_t(n)$.
\begin{theorem} \label{CW-main2}\textup{(Chan-Wang \cite[Theorem 1.2]{Chan-Wang})} Suppose $a,b,d\in \mathbb{Z}$, $b\ge 1$ and gcd$(a, b) = 1$. Let $\ell$ be a prime divisor of $a+db$ and $0 \le r<\ell$. Suppose $d, \ell$  and $r$ satisfy any of the
following conditions:
\begin{enumerate}
\item $d = 1$ and $24r + 1$ is a quadratic non-residue modulo $\ell$;
\item $d = 3$ and $8r + 1$ is a quadratic non-residue modulo $\ell$ or $8r + 1 \equiv 0~(\textup{mod}~\ell)$;
\item $d\in \{4, 8, 14\}$, $\ell \equiv 5~(\textup{mod}~6)$ and  $24r + d \equiv 0~(\textup{mod}~\ell)$;
\item $d\in \{6, 10\}$, $\ell \ge 5$ and $\ell\equiv 3~(\textup{mod}~4)$ and $24r + d  \equiv 0~(\textup{mod}~\ell)$;
\item $d = 26$, $\ell\equiv 11~(\textup{mod}~12)$  and $24r + d  \equiv 0~(\textup{mod}~\ell)$.
\end{enumerate}
Then, for $n \ge 0$, \begin{align}\label{cw-congruence} p_{-a/b}(\ell n + r)\equiv 0~(\textup{mod}~\ell).\end{align}
\end{theorem}
\noindent They \cite[Theorem 3.1 and Theorem 3.2]{Chan-Wang} listed the explicit congruences satisfied by $p_{a/b}(n)$ for $1\le |a|<b\le 5$ that follow from the above theorem. They also found three additional congruences modulo $25$ and $49$ and conjectured \cite[Conjectures 3.1 and 3.2]{Chan-Wang} seventeen more congruences. In fact, they proved one of those conjectural congruences by using the theory of modular forms and had speculated that some more congruences might be proved in a similar way.

Recenty, Xia and Zhu \cite{XiaZhu} proved many of the congruences conjectured by Chan and Wang \cite{Chan-Wang} and  also discovered new congruences
for $p_t(n)$ with the help of Ramanujan’s modular equations of fifth, seventh and thirteenth
orders. The first purpose of this paper is to show yet another elementary method that does not only prove all the conjectural congruences modulo powers of 5 by Chan and Wang, but also give us new congruences modulo higher powers of 5. The method involves dissections of $E_1$ and some  identities involving the Rogers-Ramanujan continued fraction $\mathcal{R}(q)$, defined by
\begin{align}\label{rrcf}
    \mathcal{R}(q):= \frac{q^{1/5}}{1} \cplus  \frac{q}{1}\cplus\frac{q^2}{1}\cplus  \frac{q^3}{1}\cplus\contdots, \quad |q|<1,
\end{align}
which has the following well-known $q$-product representation \cite[p. 160, Theorem 7.3.3]{Spirit}
$$\mathcal{R}(q)=q^{1/5}\dfrac{(q;q^5)_{\infty}(q^4;q^5)_{\infty}}{(q^2;q^5)_{\infty}(q^3;q^5)_{\infty}}.$$

In the next theorem, we provide a sample of five new congruences that arise due to the method we shall show.
\begin{theorem}\label{Conjecture.Theorem.mod5}  For all $n\ge 0$, we have
\begin{align*}{}
	 p_{-1/6}(25n+r)&\equiv0~(\textup{mod}~25),  \quad \quad \hspace*{2mm} r\in\{9,14,19,24\},\\
   p_{1/6}(125n+r)&\equiv0~(\textup{mod}~25), \quad \quad \hspace*{2mm} r\in\{96,121\},\\
    p_{-5/6}(125n+r)&\equiv0~(\textup{mod}~25),\quad \quad \hspace*{2mm} r\in\{95,120\}\\
  p_{5/6}(25n+r)&\equiv0~(\textup{mod}~125),\quad \quad r\in\{15,20\},\\
\intertext{and}
    p_{5/6}(125n+r)&\equiv0~(\textup{mod}~625),\quad \quad r\in\{65,70\}.
\end{align*}
\end{theorem}

\noindent This sample of congruences does not follow from \cite[Theorem 1.2]{XiaZhu} wherein Xia and Zhu proved the conjectural congruences of Chan and Wang.

Now, Dedekind's eta function is defined for $\textup{Im}~ \tau >0$ by
$$\eta(\tau)=e^{\pi i \tau/12}\prod_{j=1}^\infty (1-e^{2j \pi i\tau})= q^{1/24}\prod_{j=1}^\infty (1-q^j)=q^{1/24}E_1.$$

It follows from the classical formulas of Euler and Jacobi that $\eta(24\tau)$ and $\eta(8\tau)^3$ are lacunary. Working on this fact, Bevilacqua, Chandran and  Choi  \cite{Erin} generalized the congruence \eqref{cw-congruence} for higher powers of $\ell$ for the cases $d=1$ and $3$. They also arrived at many balanced congruences, where an $\ell^k$-balanced congruence for $k\ge1$ is of the form,
\begin{align*}
    p_t(\ell^kn+B)&\equiv0~(\textup{mod}~\ell^k), \quad \textup{for~ all}~ n\ge 0.
\end{align*}

On the other hand, for even integer $d>0$, Serre \cite{serre} proved that $\eta^d$ is lacunary if and only if
$d\in \{2,4,6,8,10,14,26\}$. Choi \cite{Choi} proved congruences for higher powers of $\ell$ for  $d\in \{2,4,6,8,10,14,26\}$ that generalizes \eqref{cw-congruence} for the cases $d\in \{4,6,8,10,14,26\}$.

 It is to be noted that very recently, Iskander, Jain, and  Talvola \cite{Iskander} found the exact formula for  $p_t(n)$ by using the circle method and studied some of its implications.

Now, from \cite[p. 274, Theorem 12.1]{NB3}, we recall the following general $n$-dissections of $E_1$ for $n\equiv \pm 1~(\textup{mod}~6)$.
\begin{lemma}\label{n-dissection}
Let $n\geq1$ be an integer with $n\equiv \pm 1~(\textup{mod~6})$. If $n=6g+1$, where $g\geq1$, then
\begin{align*}
    E_1&=E_{n^2}\Bigg((-1)^g q^{(n^2-1)/24}+\sum_{j=1}^{(n-1)/2}(-1)^{j+g}q^{(j-g)(3j-3g-1)/2}\dfrac{(q^{2jn};q^{n^2})_{\infty}(q^{n^2-2jn};q^{n^2})_{\infty}}{(q^{jn};q^{n^2})_{\infty}(q^{n^2-jn};q^{n^2})_{\infty}}\Bigg),
\end{align*}
while if $n=6g-1$, where $g\geq1$, then
\begin{align*}
   E_1&=E_{n^2}\Bigg((-1)^g q^{(n^2-1)/24}+\sum_{j=1}^{(n-1)/2}(-1)^{j+g}q^{(j-g)(3j-3g+1)/2}\dfrac{(q^{2jn};q^{n^2})_{\infty}(q^{n^2-2jn};q^{n^2})_{\infty}}{(q^{jn};q^{n^2})_{\infty}
   (q^{n^2-jn};q^{n^2})_{\infty}}\Bigg).
\end{align*}
\end{lemma}

We observe that in both the above $n$-dissections of $E_1$, there is a term $(-1)^g q^{(n^2-1)/24}$ multiplied by $E_{n^2}$ which is independent of other $q$-products. We exploit this fact to prove the following  general congruence modulo arbitrary powers of a prime $\ell\ge 5$.
\begin{theorem}\label{General.Theorem}
Let $\ell\geq5$ be a prime and $k>1$ and $s$ be positive integers such that $s\le \floor{k/2}$. Then, for all $n\ge 0$, we have
\begin{align}
     \label{General1}p_{-(\ell^{k}-b)/b}\bigg(\ell^{2s}\cdot n+\ell^{2s-1}\cdot r+\dfrac{(\ell-24\floor*{\ell/24})\ell^{2s-1}-1}{24}\bigg)&\equiv 0 ~(\textup{mod} ~\ell^{k-2s+1}),
    \end{align}
    where $0\leq r<\ell,~ r\ne\floor*{\ell/24}$ and $(\ell, b)=1$.
\end{theorem}
The above theorem generates congruences modulo arbitrary powers of any prime $\ell\ge 5$.
For example, let $\ell=5$ and $b=1567$ in the above theorem. Then by choosing $k=5$, we arrive at the following congruences
\begin{align*}
p_{-1558/1567}(5^{2} n+5r+1)&\equiv 0 ~(\textup{mod} ~5^4), \\
    \intertext{and}
    p_{-1558/1567}(5^{4} n+5^3 r+26)&\equiv 0 ~(\textup{mod} ~5^2),
        \end{align*}
where $r\in\{1,2,3,4\}$. Note that the sequences $(5^{2} n+5r+1)$ and $(5^{4} n+5^3 r+26)$  do not have common terms.

Lemma \ref{n-dissection} for the cases $n=5,7,11$ and a $3$-dissection of $E_1^3$, namely, \cite[p. 345, Entry 1(iv)]{NB3}
\begin{align}
    \label{3E_1} E_1^3&=E_9^3\Big(\dfrac{E_6E_9^3}{E_3E_{18}^3}-3q+4q^3 \dfrac{E_3^2E_{18}^6}{E_6^2E_9^6}\Big),
    \end{align}
can be exploited further for more congruences. We use  \eqref{3E_1}  and the $5$-, $7$- and $11$-dissections of $E_1^d$ for $d=2,3,4,6,8,14$, that can be derived from the corresponding formulas for $E_1$ in  Lemma \ref{n-dissection} to find more congruences for arbitrary powers of $3,5,7,$ and $11$.

The following theorem is obtained by using dissections of $E_1^2$.

\begin{theorem}\label{2powerE1-5-7-11} Let $k>1$ and $s$ be positive integers such that  $s\le\floor{k/2}$. Then, for all $n\ge 0$, we have
\begin{align}
\label{Theorem.5k.Ordinary.Congruence6} p_{-(5^{k}-2b)/b}\bigg(5^{2s}\cdot n+5^{2s-1}\cdot r+\dfrac{5^{2s}-1}{12}\bigg)&\equiv 0 ~(\textup{mod} ~5^{k-2s+1}),\quad r\in\{1,2,3,4\},\\
    \label{Theorem.7k.Ordinary.Congruence4} p_{-(7^{k}-2b)/b}\bigg(7^{2s}\cdot n+7^{2s-1}\cdot r+\dfrac{7^{2s}-1}{12}\bigg)&\equiv 0 ~(\textup{mod} ~7^{k-2s+1}),\quad r\in\{1,2,\dots,6\},\\\intertext{and}
    \label{Theorem.11k.Ordinary.Congruence2}
    p_{-(11^{k}-2b)/b}\bigg(11^{2s}\cdot n+11^{2s-1}\cdot r+\dfrac{11^{2s}-1}{12}\bigg)&\equiv 0 ~(\textup{mod} ~11^{k-2s+1}),\quad r\in\{1,2,\ldots,11\},
\end{align}
where $b$'s in the above congruences are co-prime to the moduli.
\end{theorem}

\begin{remark}\label{remark116} We notice that although the above three congruences follow a certain pattern with changes only among the primes 5, 7 and 11, yet these congruences can not be generalized for powers of all primes. In fact, it does not even hold true for the next prime 13.
\end{remark}

Our next theorem arises from dissections of $E_1^3$ and $E_1^4$.

\begin{theorem}\label{3-4powerE1-3-5-7} Let $k$, $m$ and $s$ be positive integers such that $s\leq m+1$. Then, for all $n\ge 0$, we have
\begin{align}
\label{Theorem.3k.Ordinary.Congruence2}p_{-(3^{k+m}-3b)/b}\bigg(3^{2s}\cdot n+3^{2s-1}\cdot r+\dfrac{3^{2s}-1}{8}\bigg)&\equiv 0 ~(\textup{mod} ~3^{k+m-s+1}),\quad r\in\{1,2\},\\
        \label{Theorem.5k.Ordinary.Congruence3}p_{-(5^{k+m}-3b)/b}\bigg(5^{2s}\cdot n+5^{2s-1}\cdot r+\dfrac{5^{2s}-1}{8}\bigg)&\equiv 0 ~(\textup{mod} ~5^{k+m-s+1}),\quad r\in\{1,2,3,4\},\\
        \label{Theorem.5k.Ordinary.Congruence4}p_{-(5^{k+m}-4b)/b}\bigg(5^{2s}\cdot n+5^{2s-1}\cdot r+\dfrac{5^{2s}-1}{6}\bigg)&\equiv 0 ~(\textup{mod} ~5^{k+m-s+1}),\quad r\in\{1,2,3,4\},\\\intertext{and}
    \label{Theorem.7k.Ordinary.Congruence2}p_{-(7^{k+m}-3b)/b}\bigg(7^{2s}\cdot n+7^{2s-1}\cdot r+\dfrac{7^{2s}-1}{8}\bigg)&\equiv 0 ~(\textup{mod} ~7^{k+m-s+1}),\quad r\in\{1,2,\ldots,6\},
\end{align}
where $b$'s in the above congruences are co-prime to the moduli.
\end{theorem}

The following theorem is obtained by using dissections of $E_1^6$, $E_1^8$ and $E_1^{14}$.
\begin{theorem}\label{6-8-14powerE1-3-5-7}Let $k>1$ and  $s$ be positive integers such that $ s \leq \floor{k/2}$. Then, for all $n\ge0,$ we have
\begin{align}
 \label{Theorem.3k.Ordinary.Congruence1}p_{-(3^{k}-6b)/b}\bigg(3^{2s}\cdot n+3^{2s-1}\cdot r+\dfrac{3^{2s}-1}{4}\bigg)&\equiv 0 ~(\textup{mod} ~3^k),\quad r\in\{1,2\},\\
     \label{Theorem.5k.Ordinary.Congruence1}p_{-(5^{k}-8b)/b}\bigg(5^{2s}\cdot n+5^{2s-1}\cdot r+\dfrac{2\cdot5^{2s-1}-1}{3}\bigg)&\equiv 0 ~(\textup{mod} ~5^k),\quad r\in\{0,2,3,4\},\\
         \label{Theorem.5k.Ordinary.Congruence2}p_{-(5^{k}-14b)/b}\bigg(5^{2s}\cdot n+5^{2s-1}\cdot r+\dfrac{11\cdot5^{2s-1}-7}{12}\bigg)&\equiv 0 ~(\textup{mod} ~5^k),\quad r\in\{0,1,3,4\},\\\intertext{and}
     \label{Theorem.7k.Ordinary.Congruence1} p_{-(7^{k}-6b)/b}\bigg(7^{2s}\cdot n+7^{2s-1}\cdot r+\dfrac{3\cdot7^{2s-1}-1}{4}\bigg)&\equiv 0 ~(\textup{mod} ~7^k),\quad r\in\{0,2,3,4,5,6\},
\end{align}
where $b$'s in the above congruences are co-prime to the moduli.
\end{theorem}

We also have the following additional congruences arising as the by-products of the proofs of the congruences in Theorem \ref{6-8-14powerE1-3-5-7} for odd  $k$.

\begin{theorem}\label{B-6-8-14-powerE1-3-5-7} Let $k>1$ be an odd integer. Then, for all $n\ge 0$, we have
\begin{align}
     \label{B.Theorem.3k.Ordinary.Congruence1}p_{-(3^{k}-6b)/b}\bigg(3^{k}\cdot n+\dfrac{3^{k+1}-1}{4}\bigg)&\equiv 0 ~(\textup{mod} ~3^k),\\
      \label{B.Theorem.5k.Ordinary.Congruence1}p_{-(5^{k}-8b)/b}\bigg(5^{k}\cdot n+\dfrac{2\cdot5^{k}-1}{3}\bigg)&\equiv 0 ~(\textup{mod} ~5^k),\\
    \label{B.Theorem.5k.Ordinary.Congruence2}p_{-(5^{k}-14b)/b}\bigg(5^{k}\cdot n+\dfrac{11\cdot5^{k}-7}{12}\bigg)&\equiv 0 ~(\textup{mod} ~5^k),\\\intertext{and}
    \label{B.Theorem.7k.Ordinary.Congruence1}p_{-(7^{k}-6b)/b}\bigg(7^{k}\cdot n+\dfrac{3\cdot7^{k}-1}{4}\bigg)&\equiv 0 ~(\textup{mod} ~7^k),
\end{align}
where $b$'s in the above congruences are co-prime to the moduli.
\end{theorem}

\begin{remark}\label{ord1} If we choose  $b=1$ in Theorems \ref{General.Theorem} -- \ref{B-6-8-14-powerE1-3-5-7}, then we find infinitely many congruences satisfied by  $p_{-t}(n)$, where $t>1$ is a positive integer. Recall that $p_{-t}(n)$ counts the number of $t$-colored partitions of $n$.
\end{remark}

Now, for any non-zero rational number $t$ and integer $r>1$, we define $$\sum_{n=0}^\infty p_{[1,r;t]}(n)q^n=(E_1E_r)^t.$$Clearly, $p_{[1,r;-1]}(n)$ is the number of 2-color partitions of $n$ where one of the colors appears only in parts that are multiples of $r$. It is to be noted that the $t$-color partition and the $t$-colored partition defined above are quite different.  Recently, $p_{[1,r;-1]}(n)$ and other related partition functions have been studied quite prominently. For example, see the papers of  H. C. Chan \cite{HCChan1, HCChan2, HCChan3},  H. H. Chan and Toh \cite{HHChan.PCToh},  Hirschhorn \cite{MDHirschhorn2, MDHirschhorn1}, Ahmed, Baruah and Dastidar \cite{ZAhmedNDBaruahMGDastidar}, Chern \cite{chern}, Chern and Dastidar\cite{chern-dastidar},  and Xiong \cite{XXiong}. In particular, H. C. Chan \cite{HCChan1} proved the Ramanujan-type congruence
\begin{align}\label{chan-cubic}
     p_{[1,2;-1]}(3n+2)&\equiv 0 ~(\textup{mod} ~3), \quad\textup{for~ all}~ n\ge0,
\end{align}
whereas Ahmed, Baruah and Dastidar \cite{ZAhmedNDBaruahMGDastidar} and Chern \cite{chern} established that
if \\$r\in \{2, 3, 4, 5, 7, 8, 10, 15, 17, 20\}$, then
\begin{align*}
p_{[1,r;-1]}(25n + k) \equiv 0 ~(\textup{mod} ~5), \quad\textup{for~ all}~ n\ge0,
\end{align*}
where $r+k =24$.

It can be seen from Theorem \ref{CW-main1} that denominators of both $p_{[1,r;t]}(n)$ and $t$ have the same prime divisors. For instance, we have the series expansions
\begin{align}\label{e1e2-1/4}
(E_1E_2)^{-1/4}&=1+\dfrac{1}{2^{2}}q+\dfrac{21 }{2^{5}}q^2+\dfrac{63 }{2^{7}}q^3+\dfrac{2275 }{2^{11}}q^4+\dfrac{6327 }{2^{13}}q^5+\dfrac{104657 }{2^{16}}q^6+\dfrac{311183 }{2^{18}}q^7\notag\\
&\quad+\dfrac{19341027 }{2^{23}}q^8+\dfrac{62148331}{2^{25}}q^9+O\left(q^{10}\right)\\\intertext{and}
\label{e1e3--1/6}(E_1E_3)^{1/6}&=1-\dfrac{1}{2^{}\cdot 3^{}}q-\dfrac{17 }{2^{3}\cdot 3^{2}}q^2-\dfrac{451 }{2^{4}\cdot 3^{4}}q^3-\dfrac{6191 }{2^{7}\cdot 3^{5}}q^4-\dfrac{12053 }{2^{8}\cdot 3^6}q^5-\dfrac{2845933}{2^{10}\cdot 3^8}q^6\notag\\
&\quad+\dfrac{1308439}{2^{11}\cdot 3^9}q^7-\dfrac{142565077}{2^{15}\cdot 3^{10}}q^8-
\dfrac{16863587387}{2^{16}\cdot 3^{13}}q^9+
O\left(q^{10}\right).
\end{align}
Therefore, it is also meaningful to explore congruences for $p_{[1,r;t]}(n)$, for non-integral rational numbers $t$, modulo powers of prime $\ell$ such that $\ell \nmid \textup{denom}(t)$.

We find the following three theorems that are analogous to Theorem \ref{CW-main2}.
\begin{theorem}\label{Theorem.Color12.1}
Suppose $a, b, d \in \mathbb{Z}, b\geq1$ and $(a, b) = 1$. Let $\ell$ be an odd prime divisor of $a+db$ and $0\leq r <\ell$. Suppose $d,\ell$ and $r$ satisfy any of the following two conditions:
\begin{itemize}
    \label{Color12.Condition1}\item[1.] $d=2$, $\ell\equiv3~(\textup{mod}~4)$ and $4r+1\equiv0~(\textup{mod}~\ell)$,
    \label{Color12.Condition2}\item[2.] $d=3$, $\ell\equiv5~or~7~(\textup{mod}~8)$ and $8r+3\equiv0~(\textup{mod}~\ell)$.
 \end{itemize}
Then, for all $n\ge0$,
\begin{align}
    \label{General.Congruence.1}p_{[1,2;-a/b]}(\ell n+r)&\equiv0~(\textup{mod}~\ell).
\end{align}
\end{theorem}

Note that H. C. Chan's congruence \eqref{chan-cubic} follows immediately from the above theorem when we choose $a=b=1$, $\ell=3$ and $r=2$. If $a=1$, $b=4$, $\ell=3$ and $r=2$ in \eqref{General.Congruence.1}, then $$p_{[1,2;-1/4]}(3 n+2)\equiv0~(\textup{mod}~3),$$ and we see that the series expansion \eqref{e1e2-1/4} of $(E_1E_2)^{-1/4}$ demonstrates the first few cases of this congruence.

\begin{theorem}\label{Theorem.Color13.1}
Suppose $a, b, d \in \mathbb{Z}, b\geq1$ and $(a, b) = 1$. Let $\ell$ be an odd prime divisor of $a+db$ and $0\leq r <\ell$. Suppose $d,\ell$ and $r$ satisfy  the following condition:
\begin{align*}
d=3, \ell\equiv5~or~11~(\textup{mod}~12) ~and~ 2r+1\equiv0~(\textup{mod}~\ell).
\end{align*}
Then, for all $n\ge0$,
\begin{align}
    \label{General.Congruence.13.1}p_{[1,3;-a/b]}(\ell n+r)&\equiv0~(\textup{mod}~\ell).
\end{align}
\end{theorem}

When we choose $a=-1$, $b=6$, $\ell=17$ and $r=8$ in the above theorem, then it follows that $$p_{[1,3;1/6]}(17n+8)\equiv0~(\textup{mod}~17),$$ and the series expansion \eqref{e1e3--1/6} of $(E_1E_3)^{1/6}$ demonstrates the first case $n=0$.

\begin{theorem}\label{Theorem.Color14.1}
Suppose $a, b, d \in \mathbb{Z}, b\geq1$ and $(a, b) = 1$. Let $\ell$ be an odd prime divisor of $a+db$  and $0\leq r <\ell$. Suppose $d,\ell$ and $r$ satisfy any of the
following two conditions:
\begin{itemize}
    \label{Color14.Condition1}\item[1.] $d=2$, $\ell\equiv3~(\textup{mod}~4)$ and $12r+5\equiv0~(\textup{mod}~\ell)$,
    \label{Color14.Condition2}\item[2.] $d=3$, $\ell\equiv3~(\textup{mod}~4)$ and $8r+5\equiv0~(\textup{mod}~\ell)$.
 \end{itemize}
Then, for all $n\ge0$,
\begin{align}
    \label{General.Congruence.2}p_{[1,4;-a/b]}(\ell n+r)&\equiv0~(\textup{mod}~\ell).
\end{align}
\end{theorem}

Theorem \ref{Theorem.Color14.1} does not furnish any congruence modulo 5. But, with the aid of a  5-dissection of $E_1$, we obtain the following result.
\begin{theorem}\label{Theorem.Color14.2} For integer $k\ge 1$ and all $n\ge 0$, we have
\begin{align}
    \label{Theorem.5k.Color14.2.Congruence}p_{[1,4;-(5^k-3b)/b]}(5 n+r)&\equiv 0 ~(\textup{mod} ~5), \quad where~(5,b)=1~and~ r\in\{2,3\}.
\end{align}
\end{theorem}

\begin{remark}\label{b-equal1} If we choose  $b=1$ in \eqref{Theorem.5k.Color14.2.Congruence} so that $(5^k-3b)/b=5^k-3=:N>1$, then we arrive at
\begin{align*}
    p_{[1,4;-N]}(5 n+r)&\equiv 0 ~(\textup{mod} ~5), \quad r\in\{2,3\},
\end{align*}
where $p_{[1,4;-N]}(n)$ counts the number of $2N$-color partitions of $n$ with $N$
of the colors appearing only in multiples of $4$.
\end{remark}

Next, using the dissections of $(E_1E_r)^d$ for $r=2,3$, and $4$ and $d=1,2,3$, and $5$, we find several theorems analogous to Theorems \ref{General.Theorem} -- \ref{B-6-8-14-powerE1-3-5-7}.

The following theorem arises from the dissections of $E_1E_r$ for $r=2,3$, and $4$.

\begin{theorem}\label{Theorem.(m+1)/2}
Let $k>1$ and $s$ be positive integers such that $s\leq \floor{k/2}$. Then, for all $n\ge0,$ we have
\begin{align}
    \label{Theorem.3k.Color12.Congruence3} p_{[1,2;-(3^{k}-b)/b]}\bigg(3^{2s}\cdot n+3^{2s-1}\cdot r+\dfrac{3^{2s}-1}{8}\bigg)&\equiv 0 ~(\textup{mod} ~3^{k-2s+1}),\quad r\in\{1,2\},\\
           \label{Theorem.5k.Color12.Congruence4} p_{[1,2;-(5^{k}-b)/b]}\bigg(5^{2s}\cdot n+5^{2s-1}\cdot r+\dfrac{5^{2s}-1}{8}\bigg)&\equiv 0 ~(\textup{mod} ~5^{k-2s+1}),\quad r\in\{1,2,3,4\},\\
           \label{Theorem.7k.Color12.Congruence1} p_{[1,2;-(7^{k}-b)/b]}\bigg(7^{2s}\cdot n+7^{2s-1}\cdot r+\dfrac{7^{2s}-1}{8}\bigg)&\equiv 0 ~(\textup{mod} ~7^{k-2s+1}),\quad r\in\{1,2,\ldots,6\},\\
    \label{Theorem.5k.Color13.Congruence1Duplicate} p_{[1,3;-(5^{k}-b)/b]}\bigg(5^{2s}\cdot n+5^{2s-1}\cdot r+\dfrac{5^{2s}-1}{6}\bigg)&\equiv 0 ~(\textup{mod} ~5^{k-2s+1}),\quad r\in\{1,2,3,4\},\\
    \label{Theorem.7k.Color14.Congruence1} p_{[1,4;-(7^{k}-b)/b]}\bigg(7^{2s}\cdot n+7^{2s-1}\cdot r+\dfrac{11\cdot7^{2s-1}-5}{24}\bigg)&\equiv 0 ~(\textup{mod} ~7^{k-2s+1}), \quad r\in\{0,2,3,\dots,6\},\\
       \label{Theorem.11k.Color13.Congruence1}
    p_{[1,3;-(11^{k}-b)/b]}\bigg(11^{2s}\cdot n+11^{2s-1}\cdot r+\dfrac{5\cdot11^{2s-1}-1}{6}\bigg)&\equiv 0 ~(\textup{mod} ~11^{k-2s+1}),\quad r\in\{0,2,3,\ldots,10\},\\\intertext{and}
     \label{Theorem.11k.Color14.Congruence1}
  p_{[1,4;-(11^{k}-b)/b]}\bigg(11^{2s}\cdot n+11^{2s-1}\cdot r+\dfrac{7\cdot11^{2s-1}-5}{24}\bigg)&\equiv 0 ~(\textup{mod} ~11^{k-2s+1}),\quad r\in\{0,1,3,4,\dots,10\},
\end{align}
where $b$'s in the above congruences are co-prime to the moduli.
\end{theorem}

\begin{remark} As in Remark $\ref{remark116},$ we notice that the congruences \eqref{Theorem.3k.Color12.Congruence3}, \eqref{Theorem.5k.Color12.Congruence4} and \eqref{Theorem.7k.Color12.Congruence1} follow a common pattern. But the pattern does not hold true for $11$.
\end{remark}

The next theorem arises from the dissections of $(E_1E_r)^2$ for $r=2$ and 3.

\begin{theorem}\label{Theorem.m}
Let  $k$, $m,$ and $s$ be positive integers such that $s\leq m+1$. Then, for all $n\ge0,$ we have
\begin{align}
         \label{Theorem.3k.Color12.Congruence2} p_{[1,2;-(3^{k+m}-2b)/b]}\bigg(3^{2s}\cdot n+3^{2s-1}\cdot r+\dfrac{3^{2s}-1}{4}\bigg)&\equiv 0 ~(\textup{mod} ~3^{k+m-s+1}),\quad r\in\{1,2\}\\\intertext{and}
        \label{Theorem.5k.Color13.Congruence2}p_{[1,3;-(5^{k+m}-2b)/b]}\bigg(5^{2s}\cdot n+5^{2s-1}\cdot r+\dfrac{2\cdot 5^{2s-1}-1}{3}\bigg)&\equiv 0 ~(\textup{mod} ~5^{{k+m-s+1}}),\quad r\in\{0,2,3,4\},
   \end{align}
   where $b$'s in the above congruences are co-prime to the moduli.
\end{theorem}

The congruences in the next theorem arise from the dissections of $(E_1E_r)^3$ for $r=2$ and 3 and $(E_1E_2)^5$.

\begin{theorem}\label{Theorem.k/2}
Let $k>1$ and  $s$ be positive integers such that $ s \leq \floor{k/2}$. Then, for all $n\ge0,$ we have
\begin{align}
            \label{Theorem.3k.Color12.Congruence1} p_{[1,2;-(3^{k}-5b)/b]}\bigg(3^{2s}\cdot n+3^{2s-1}\cdot r+\dfrac{7\cdot3^{2s-1}-5}{8}\bigg)&\equiv 0 ~(\textup{mod} ~3^k),\quad r\in\{0,2\},\\
             \label{Theorem.5k.Color12.Congruence1}p_{[1,2;-(5^{k}-3b)/b]}\bigg(5^{2s}\cdot n+5^{2s-1}\cdot r+\dfrac{7\cdot5^{2s-1}-3}{8}\bigg)&\equiv 0 ~(\textup{mod} ~5^k),\quad r\in\{0,2,3,4\},\\\intertext{and}
     \label{Theorem.5k.Color13.Congruence1}p_{[1,3;-(5^{k}-3b)/b]}\bigg(5^{2s}\cdot n+5^{2s-1}\cdot r+\dfrac{5^{2s-1}-1}{2}\bigg)&\equiv 0 ~(\textup{mod} ~5^k),\quad r\in\{0,1,3,4\},
    \end{align}
    where $b$'s in the above congruences are co-prime to the moduli.
\end{theorem}

Finally, we have the following balanced congruences arising as the by-products of the proofs of the congruences in Theorem \ref{Theorem.k/2} for odd $k$.
\begin{theorem}\label{balancedtheorem}
Let $k>1$ be an odd integer. Then, for all $n\ge 0$, we have
\begin{align}
           \label{B.Theorem.3k.Color12.Congruence1} p_{[1,2;-(3^{k}-5b)/b]}\bigg(3^{k}\cdot n+\dfrac{7\cdot3^{k}-5}{8}\bigg)&\equiv 0 ~(\textup{mod} ~3^k),\\
         \label{B.Theorem.5k.Color12.Congruence1}p_{[1,2;-(5^{k}-3b)/b]}\bigg(5^{k}\cdot n+\dfrac{7\cdot5^{k}-3}{8}\bigg)&\equiv 0 ~(\textup{mod} ~5^k),\\\intertext{and}
     \label{B.Theorem.5k.Color13.Congruence1}p_{[1,3;-(5^{k}-3b)/b]}\bigg(5^{k}\cdot n+\dfrac{5^{k}-1}{2}\bigg)&\equiv 0 ~(\textup{mod} ~5^k),
    \end{align}
    where $b$'s in the above congruences are co-prime to the moduli.
\end{theorem}

\begin{remark} As in Remark \ref{b-equal1}, it is worthwhile to note that setting $b=1$ in Theorems \ref{Theorem.(m+1)/2} --\ref{balancedtheorem}, we can deduce corresponding congruences for $2N$-color partition function $p_{[1,r;-N]}(n)$, where $N>1$.
\end{remark}

Newman \cite{Newman} and Farkas and Kra \cite{Farkas} found several two-term  and three-term recurrence relations satisfied by $p_t(n)$ for some positive integers $t$. More such  recurrence relations are found by Baruah and Sarmah \cite{baruah-sarmah}. In particular, they proved thirteen two-term recurrence relations satisfied by $p_t(n)$, where $t\in \{1,2,3,4,6,8,14\}$.  It may be noted that those relations are the genesis of Theorems \ref{2powerE1-5-7-11} -- \ref{B-6-8-14-powerE1-3-5-7}. We found similar recurrence relations for $p_{[1,r;t]}$, where $r\in\{2,3,4\}$ and  $t\in \{1,2,3,5\}$. These are the basis for Theorems \ref{Theorem.(m+1)/2} -- \ref{balancedtheorem}. It would be interesting to investigate such recurrences for more values of $r$ and $t$ that may lead to more congruences for fractional partition functions.

We organize the paper in the following way. In the next section, we present some preliminary results and set-up. In Section \ref{sec3}, we prove Theorem \ref{Conjecture.Theorem.mod5} which contains ten of the seventeen conjectural congruences of Chan and Wang \cite{Chan-Wang}. In Section \ref{sec4}, we prove  Theorems \ref{General.Theorem}--\ref{B-6-8-14-powerE1-3-5-7} that give congruences for $p_t(n)$. In Sections \ref{sec5} and \ref{sec6}, we prove Theorems \ref{Theorem.Color12.1}--\ref{Theorem.Color14.2} and Theorems \ref{Theorem.(m+1)/2}--\ref{balancedtheorem}, respectively, that contain congruences for $p_{[1,r;t]}(n)$.

\section{\bf Preliminaries} \label{Preliminaries}

We will frequently use the following lemma in the subsequent
sections, sometimes without referring to it.

 \begin{lemma}\label{UnderModulo}
 Let $k=a/b$, where $a,b \in \mathbb{Z}, b\geq1$ and $(a,b)=1$. Let $\ell$ be a prime such that $\ell\nmid b$. Then for any positive integers $j$ and $n$
\begin{align*}
    E_n^{\ell^j k}\equiv E_{\ell n}^{\ell^{j-1}k}~(\textup{mod}~\ell^j).
\end{align*}
 \end{lemma}
 \begin{proof}
 See \cite[p. 64, Lemma 2.1]{Chan-Wang}.
 \end{proof}

A $3$-dissection of $E_1E_2$ is given in the next lemma.
\begin{lemma}\label{Dissections}
We have
\begin{align}
    \label{3E_1E_2} E_1E_2&=E_9E_{18}\bigg(\dfrac{E_6E_9^3}{E_3E_{18}^3}-q+2q^2 \dfrac{E_3E_{18}^3}{E_6E_9^3}\bigg).
  \end{align}
\end{lemma}

\begin{proof}
See \cite[p. 132]{MDHirschhorn2}.
\end{proof}

Our next lemma gives some well-known series expansions.
\begin{lemma}\label{SeriesExpansion}
We have
\begin{align}
    \label{PentagonalTheorem}E_1&=\sum_{n=-\infty}^{\infty}(-1)^n q^{((6n+1)^2-1)/24},\\
    \label{JacobiTheorem}E_1^3&=\sum_{n=-\infty}^{\infty}(4n+1) q^{((4n+1)^2-1)/8},\\
    \label{PsiFunction}\dfrac{E_2^2}{E_1}&=\sum_{n=-\infty}^{\infty} q^{((4n+1)^2-1)/8},\\
    \label{FoundInHirschhorn}\dfrac{E_1^2E_4^2}{E_2}&=\sum_{n=-\infty}^{\infty}(3n+1)q^{((3n+1)^2-1)/3}.
\end{align}
\end{lemma}
\begin{proof}
See \cite[pp. 15--17]{Cooper}.
\end{proof}

Identities involving the Rogers-Ramanujan continued fraction are used in the proofs of the theorems. In the next two lemmas, we present some of those identities.  In fact, the first lemma is the case $n=5$ of Lemma \ref{n-dissection}.

\begin{lemma}\label{E1-Rq5} Let $$R(q):=\dfrac{q^{1/5}}{\mathcal{R}(q)},$$
where $\mathcal{R}(q)$ is the Rogers-Ramanujan continued fraction as defined in \eqref{rrcf}. We have the following $5$-dissection of $E_1$:

    \begin{align}
       \label{E1R} E_1= E_{25}\left(R(q^5)-q-\dfrac{q^2}{R(q^5)}\right).
    \end{align}
\end{lemma}

\begin{lemma}\label{ContinuedFractionIdentities}
We have
    \begin{align}
       \label{Identity1} &R^5(q)-\dfrac{q^2}{R^5(q)}=11q+\dfrac{E_1^6}{E_5^6},\\
        \label{Identity2}R^2(q) R(q^3)-&\dfrac{R^2(q^3)}{R(q)}+q^2\dfrac{R(q)}{R^2(q^3)}-\dfrac{q^2}{R^2(q) R(q^3)} =3q.
    \end{align}
\end{lemma}

\begin{proof}
See \cite[Chapter 7, Theorem 7.4.4]{Spirit} and \cite[pp. 193--194]{ZAhmedNDBaruahMGDastidar} for the proofs of \eqref{Identity1} and \eqref{Identity2},  respectively.
\end{proof}

If we set
\begin{align}\label{x_k}x_k:=R^{5k}(q)+(-1)^k\dfrac{q^{2k}}{R^{5k}(q)},~ \textup{for}~k\ge1,\end{align} then
\eqref{Identity1} is
\begin{alignat}{2}
   \label{Identity3new}x_1&=11q+\dfrac{E_1^6}{E_5^6},
\end{alignat}
and hence,
\begin{alignat}{2}
       \label{Identity3}x_2&=123q^2+22q\dfrac{E_1^6}{E_5^6}+\dfrac{E_1^{12}}{E_5^{12}}.
\end{alignat}
%For a proof of the identity \eqref{Identity1}, see \cite[Chapter 7, Theorem 7.4.4]{Spirit}.
For $k\geq3$, we have the following recurrence relation
\begin{align}
    \label{Recurrence}x_k=x_1x_{k-1}+q^2x_{k-2}.
\end{align}

Finally, we end this section by stating some more basic set-up. For a power series $\displaystyle{\sum_{n=0}^\infty A(n)q^n}$ and integers $0\leq r < s$, we define the extraction operator $\big[q^{s n+r}\big]$ as
\begin{align*}
    \big[q^{sn+r}\big]\Bigg\{\sum_{n=0}^\infty A(n)q^n\Bigg\}&=\sum_{n=0}^\infty A(s n+r)q^n,
\end{align*}
and for integers $d$, $a$ and $b$ such that gcd$(a,b)=1$, we have
\begin{align}
    \label{arrangement1}\sum_{n=0}^\infty p_{-a/b}(n)q^n&=E_1^{-a/b}=\dfrac{E_1^d}{E_1^{(a+db)/b}},\\\intertext{and}
    \label{arrangement2}\sum_{n=0}^\infty p_{[1,r;-a/b]}(n)q^n&=(E_1E_r)^{-a/b}=\dfrac{(E_1E_r)^d}{(E_1E_r)^{(a+db)/b}}.
\end{align}

\section{\bf Proof of Theorem \ref{Conjecture.Theorem.mod5}} \label{sec3}

\begin{proof}[Proof of Theorem \ref{Conjecture.Theorem.mod5}]
The proofs of all the congruences are similar in nature. Therefore, we prove only the fourth one in detail. By Lemma \ref{UnderModulo}, we have
\begin{align*}
    \sum_{n=0}^\infty p_{5/6}(n)q^n=E_1^{5/6}&=\dfrac{E_1^{125/6}E_1^{105}}{E_1^{125}}\equiv\dfrac{E_5^{25/6}E_1^{105}}{E_5^{25}}~(\textup{mod}~125).
\end{align*}
Employing  the 5-dissection of $E_1$ from \eqref{E1R} in the above and then applying $\big[q^{5n}\big]$, we obtain
\begin{align}
\label{Extraction1*}\sum_{n=0}^\infty p_{5/}(5n)q^n&\equiv\dfrac{E_1^{25/6}E_5^{105}}{E_1^{25}}\big(x_{21}-106qx_{20}+14q^2x_{19}-110q^3x_{18}-70q^4x_{17}+76q^5x_{16}
\notag\\
&\quad-36q^6x_{15}-71q^7x_{14}+60q^8x_{13}-80q^9x_{12}+27q^{10}x_{11}
-57q^{11}x_{10}+13q^{12}x_{9}\notag\\
&\quad -40q^{13}x_{8}+45q^{14}x_{7}+13q^{15}x_{6}-23q^{16}x_{5}
+17q^{17}x_{4}+75q^{18}x_{3}-100q^{19}x_{2}\notag\\
&\quad+45q^{20}x_{1}+80\big)~(\textup{mod}~125),
\end{align}
where $x_k$'s are as given by \eqref{x_k}. We employ \eqref{Identity3new}, \eqref{Identity3}, and the recursion \eqref{Recurrence} in \textit{Wolfram's Mathematica} to find all the $x_k$'s in \eqref{Extraction1*}. For example, we have
\begin{align*}
    x_3&=1364 q^3+\dfrac{E_1^{18}}{E_5^{18}}+33q\dfrac{ E_1^{12}}{E_5^{12}}+366q^2\dfrac{E_1^6}{E_5^6}.
\end{align*}
Having found the $x_k$'s, and then putting them in \eqref{Extraction1*}, we find that
\begin{align*}
   \sum_{n=0}^\infty p_{5/6}(5n)q^n&\equiv E_1^{125/6}E_1+80\dfrac{E_5^{105}}{E_1^{125/6}}-80q^{21}\dfrac{E_5^{105}}{E_1^{125/6}} ~(\textup{mod}~125),
\end{align*}
which, by \eqref{E1R}, can be rewritten as
\begin{align*}
   \sum_{n=0}^\infty p_{5/6}(5n)q^n&\equiv E_5^{25/6}E_{25}\left(R(q^5)-q-\dfrac{q^2}{R(q^5)}\right)+80\dfrac{E_5^{105}}{E_5^{25/6}}-80q^{21}\dfrac{E_5^{105}}{E_5^{25/6}} ~(\textup{mod}~125).
\end{align*}
Applying $\big[q^{5n+3}\big]$ and $\big[q^{5n+4}\big]$ in the above, we arrive at the fourth congruence of Theorem \ref{Conjecture.Theorem.mod5}.
\end{proof}

\section{\bf Proofs of Theorems \ref{General.Theorem}--\ref{B-6-8-14-powerE1-3-5-7}} \label{sec4}

\begin{proof}[Proof of Theorem \ref{General.Theorem}] First, note that for any prime $ \ell >24$, we have
$$\dfrac{\ell^2-1}{24}>\ell.$$
 Now, since $$\dfrac{\ell^2-24\ell\floor*{\ell/24}-1}{24}=\ell\bigg(\dfrac{\ell}{24}-\floor*{\dfrac{\ell}{24}}\bigg)-\dfrac{1}{24},$$ we see that $\floor*{\ell/24}$ is the least integer such that
$$\dfrac{\ell^2-24\ell\floor*{\ell/24}-1}{24}<\ell$$ is true. Therefore, the extraction operator $\big[q^{\ell n+(\ell^2-24\ell\floor*{\ell/24}-1)/24}\big]$ is valid for any prime $\ell\geq5$.

By $\ell$-dissection of $E_1$ from Lemma \ref{n-dissection}, we have
$$\dfrac{(j-g)(3j-3g\pm1)}{2}\neq\dfrac{\ell^2-1}{24},$$
for any $j$ such that $1\leq j \leq (\ell-1)/2$  and $\ell=6g\pm1$. Therefore, the powers of $q$ are not of the form $\ell n+(\ell^2-1)/24$ in the following sum
\begin{align}
\label{sum}\sum_{j=1}^{(\ell-1)/2}(-1)^{j+g}q^{(j-g)(3j-3g\pm1)/2}\dfrac{(q^{2j\ell};q^{\ell^2})_{\infty}(q^{\ell^2-2j\ell};q^{\ell^2})_{\infty}}{(q^{j\ell};q^{\ell^2})_{\infty}(q^{\ell^2-j\ell};q^{\ell^2})_{\infty}}.
\end{align}
In particular, powers of $q$ of the form $\ell n+(\ell^2-1)/24-\ell\floor*{\ell/24}=\ell(n-\floor*{\ell/24})+(\ell^2-1)/24$ are not present in \eqref{sum}.

%The proof is similar to that of Theorem \ref{Theorem.(m+1)/2} except the explanation on the constraints imposed on the prime $\ell$. We only provide the important data for the proof.
%Note that in either dissection, First, we discuss the constraints imposed on $\ell$.
Now, we choose $d=1$ and $a=\ell^{k}-bd=\ell^{k}-b$ in \eqref{arrangement1}, where $k\geq1$ and $(\ell^{k}-b,b)=1$. Then, by Lemma \ref{UnderModulo}, we have
%with initial extraction $\big[q^{\ell n+(\ell^2-24\ell\floor*{\ell/24}-1)/24}\big]$ along the way.
\begin{align}
   \label{GenStart} \sum_{n=0}^\infty p_{-(\ell^{k}-b)/b}(n)q^n&=\dfrac{E_1}{E_1^{\ell^{k}/b}}\equiv\dfrac{E_1}{E_{\ell}^{\ell^{k-1}/b}}  ~(\textup{mod}~\ell^{k-2s+1}),
\end{align}
where $s\ge1$ is an integer. Employing $\ell$-dissection of $E_1$ given by Lemma \ref{n-dissection} in the above, and then applying the operator $\big[q^{\ell n+(\ell^2-24\ell\floor*{\ell/24}-1)/24}\big]$, we obtain
\begin{align}
     \label{GenEnd}\sum_{n=0}^\infty p_{-(\ell^{k}-b)/b}\bigg(\ell n+\dfrac{\ell^2-24\ell\floor*{\ell/24}-1}{24}\bigg)q^n&\equiv (-1)^g q^{\floor*{\ell/24}}\dfrac{E_{\ell}}{E_1^{\ell^{k-1}/b}} \notag\\
     &\equiv (-1)^g q^{\floor*{\ell/24}}\dfrac{E_{\ell}}{E_{\ell}^{\ell^{k-2}/b}}  ~(\textup{mod}~\ell^{k-2s+1}),
\end{align}
which yields
\begin{align}
    \label{GenCongruenceSample1} p_{-(\ell^{k}-b)/b}\bigg(\ell^2 n+\ell r+\dfrac{\ell^2-24\ell\floor*{\ell/24}-1}{24}\bigg)&\equiv 0 ~(\textup{mod}~\ell^{k-2s+1}), \quad 0\leq r<\ell,~r\ne\floor*{\ell/24},
\end{align}and when $s=1$, we obtain the highest possible modulus for the above congruence which is the case $s=1$ of \eqref{General1}.

Next, operating $\big[q^{\ell n+\floor*{\ell/24}}\big]$ in \eqref{GenEnd}, we obtain
\begin{align*}
     \sum_{n=0}^\infty p_{-(\ell^{k}-b)/b}\bigg(\ell^2n+\ell \floor*{\ell/24}+\dfrac{\ell^2-24\ell\floor*{\ell/24}-1}{24}\bigg)q^n&\equiv (-1)^g \dfrac{E_{1}}{E_{1}^{\ell^{k-2}/b}} ~(\textup{mod}~\ell^{k-2s+1}).
\end{align*}
Now, in order to proceed further, we need to take $s=2$ in the above. The procedure from \eqref{GenStart} to \eqref{GenEnd} can be iterated now in the above. In fact, such procedures can be iterated til we take $s=\floor{k/2}$ times. After the $\floor{k/2}$th iteration, we have
\begin{align*}
    \sum_{n=0}^\infty p_{-(\ell^{k}-b)/b}\bigg(\ell^{2\floor{k/2}-1} n+\dfrac{(\ell-24\floor*{\ell/24})\ell^{2\floor{k/2}-1}-1}{24}\bigg)q^n&\equiv(-1)^{g\floor{k/2}} q^{\floor*{\ell/24}}\dfrac{E_{\ell}}{E_{\ell}^{\ell^{k-2\floor{k/2}}/b}}\\  
    &~\quad(\textup{mod}~\ell^{k-2s+1}).
\end{align*}
which gives
\begin{alignat}{2}
    \label{GenCongruenceSample2} p_{-(\ell^{k}-b)/b}\bigg(\ell^{2\floor{k/2}}\cdot n+\ell^{2\floor{k/2}-1}\cdot r+&\dfrac{(\ell-24\floor*{\ell/24})\ell^{2\floor{k/2}-1}-1}{24}\bigg)&&\equiv 0 ~(\textup{mod}~\ell^{k-2s+1}),
\end{alignat}
 where $0 \leq r<\ell,~r\ne\floor*{\ell/24}$, and this gives us the case $s=\floor{k/2}$ of \eqref{General1}. Congruences \eqref{GenCongruenceSample1} and \eqref{GenCongruenceSample2} along with the others found in the intermediate steps imply \eqref{General1}.
\end{proof}

\begin{proof}[Proof of Theorem \ref{2powerE1-5-7-11}]
The proofs of all the congruences are similar in nature. Therefore, we prove only congruence \eqref{Theorem.5k.Ordinary.Congruence6}. We first choose $d=2$ and $a=5^{k}-bd=5^{k}-2b$ in \eqref{arrangement1}, where  $k>1$ such that $(5^{k}-2b,b)=1$. Then, by Lemma \ref{UnderModulo}, we have
\begin{align}
    \label{SecondProofStart-18Nov}\sum_{n=0}^\infty p_{-(5^{k}-2b)/b}(n)q^n&=\dfrac{E_1^2}{E_1^{5^{k}/b}}\equiv  \dfrac{E_1^2}{E_5^{5^{k-1}/b}}~(\textup{mod}~5^{k-2s+1}),
\end{align}
where $s\ge1$ is an integer. With the aid of \eqref{E1R}, we find that
\begin{align}
\label{Similar-Identity1-18Nov}\big[q^{5n+2}\big]\left\{E_1^2\right\}=-E_5^2.
\end{align}
Applying $\big[q^{5n+2}\big]$ in \eqref{SecondProofStart-18Nov} and then employing  \eqref{Similar-Identity1-18Nov}, we have
\begin{align}
   \label{SecondProofEnd-18Nov} \sum_{n=0}^\infty p_{-(5^{k}-2b)/b}(5n+2)q^n&\equiv-\dfrac{E_5^2}{E_1^{5^{k-1}/b}}\notag\\
    &\equiv-\dfrac{E_5^2}{E_5^{5^{k-2}/b}}  ~(\textup{mod}~5^{k-2s+1}),
\end{align}
which yields
\begin{align}
    \label{SecondCongruenceSample1-18Nov}p_{-(5^{k}-2b)/b}(5^2n+5r+2)&\equiv 0 ~(\textup{mod}~5^{k-2s+1}), \quad r\in\{1,2,3,4\},
\end{align} and when $s=1$, we obtain the highest possible modulus for the above congruence which is
the case $s = 1$ of \eqref{Theorem.5k.Ordinary.Congruence6}.

Again, applying $\big[q^{5n}\big]$ in \eqref{SecondProofEnd-18Nov}, we obtain
\begin{align*}
     \sum_{n=0}^\infty p_{-(5^{k}-2b)/b}(5^2n+2)q^n
     &\equiv-\dfrac{E_1^2}{E_1^{5^{k-2}/b}}  ~(\textup{mod}~5^{k-2s+1}).
\end{align*}
In order to proceed further, we need to take s = 2 in the above, so that the procedure from \eqref{SecondProofStart-18Nov} to \eqref{SecondProofEnd-18Nov} can be iterated now in the above. The same procedure can be iterated til we take $s=\floor{k/2}$. After the $\floor{k/2}$th iteration, we arrive at
\begin{align*}
    \sum_{n=0}^\infty p_{-(5^{k}-2b)/b}\bigg(5^{2\floor{k/2}-1}\cdot n+\dfrac{5^{2\floor{k/2}}-1}{12}\bigg)q^n&\equiv(-1)^{\floor{k/2}}\dfrac{E_5^2}{E_5^{5^{k-2\floor{k/2}}/b}} ~(\textup{mod}~5^{k-2s+1}).
\end{align*}
which implies that
\begin{align}
    \label{SecondCongruenceSample2-18Nov} p_{-(5^{k}-2b)/b}\bigg(5^{2\floor{k/2}}\cdot n+5^{2\floor{k/2}-1}\cdot r+\dfrac{5^{2\floor{k/2}}-1}{12}\bigg)&\equiv 0 ~(\textup{mod}~5^{k-2s+1}),\quad r\in\{1,2,3,4\},
    \end{align}
 which is the case $s=\floor{k/2}$ of \eqref{Theorem.5k.Ordinary.Congruence3}. Congruences \eqref{SecondCongruenceSample1-18Nov} and \eqref{SecondCongruenceSample2-18Nov} along with the others found in the intermediate steps imply \eqref{Theorem.5k.Ordinary.Congruence6}.

 The following table contains the important identities required to prove the other congruences.
\begin{center}
 \begin{tabular}{c c}
 \hline
 Congruences  & Corresponding identity similar to \eqref{Similar-Identity1-18Nov}\T\B\\  \hline
 %\eqref{Theorem.5k.Ordinary.Congruence6} &  $\big[q^{5 n+2}%\big]\{E_1^2\}=-E_5^2$\T\B\\
  \eqref{Theorem.7k.Ordinary.Congruence4}& $\big[q^{7 n+4}\big]\{E_1^2\}=E_7^2$\T\B\\
  \eqref{Theorem.11k.Ordinary.Congruence2}& $\big[q^{11 n+10}\big]\{E_1^2\}=E_{11}^2$\T\B\\\hline
\end{tabular}
\end{center}
\end{proof}

\begin{proof}[Proof of Theorem \ref{3-4powerE1-3-5-7}]
The proofs of all the congruences are similar in nature. Therefore, we prove only congruence \eqref{Theorem.5k.Ordinary.Congruence3}. We choose $d=3$ and $a=5^{k+m}-bd=5^{k+m}-3b$ in \eqref{arrangement1}, where $m\geq0$ and $k\geq1$ such that $(5^{k+m}-3b,b)=1$. By Lemma \ref{UnderModulo}, we have
\begin{align}
    \label{SecondProofStart}\sum_{n=0}^\infty p_{-(5^{k+m}-3b)/b}(n)q^n&=\dfrac{E_1^3}{E_1^{5^{k+m}/b}}\equiv  \dfrac{E_1^3}{E_5^{5^{k-1+m}/b}}~(\textup{mod}~5^{k+m-s+1}),
\end{align}
where $s\ge1$ is an integer. Using \eqref{E1R}, we find that
\begin{align}
\label{Similar-Identity1}\big[q^{5n+3}\big]\left\{E_1^3\right\}=-5E_5^3.
\end{align}
Applying $\big[q^{5n+3}\big]$ in \eqref{SecondProofStart} and then employing  \eqref{Similar-Identity1}, we obtain
\begin{align}
   \label{SecondProofEnd} \sum_{n=0}^\infty p_{-(5^{k+m}-3b)/b}(5n+3)q^n&\equiv-5\dfrac{E_5^3}{E_1^{5^{k-1+m}/b}}\notag\\
    &\equiv-5\dfrac{E_5^3}{E_5^{5^{k-1+m-1}/b}}  ~(\textup{mod}~5^{k+m-s+1}),
\end{align}
which implies that
\begin{align}
    \label{SecondCongruenceSample1}p_{-(5^{k+m}-3b)/b}(5^2n+5r+3)&\equiv 0 ~(\textup{mod}~5^{k+m-s+1}), \quad r\in\{1,2,3,4\},
\end{align} and when $s = 1,$ we obtain the highest possible modulus for the above congruence which is
the case  $s=1$ of \eqref{Theorem.5k.Ordinary.Congruence3}.

Next, operating $\big[q^{5n}\big]$ in \eqref{SecondProofEnd}, we find that
\begin{align*}
     \sum_{n=0}^\infty p_{-(5^{k+m}-3b)/b}(5^2n+3)q^n&\equiv-5\dfrac{E_1^3}{E_1^{5^{k-1+m-1}}}  ~(\textup{mod}~5^{k+m-s+1}).
\end{align*}
Now, in order to proceed further, we need to take s = 2 in the above so that the procedure from \eqref{SecondProofStart} to \eqref{SecondProofEnd} can be iterated in the above. Applying the same procedure when we take $s=m+1$, we obtain
\begin{align}
   \label{ForRemark} \sum_{n=0}^\infty p_{-(5^{k+m}-3b)/b}\bigg(5^{2m+1}\cdot n+\dfrac{5^{2(m+1)}-1}{8}\bigg)q^n&\equiv(-5)^{m+1}\dfrac{E_5^3}{E_5^{5^{k-1-m-1/b}}} ~(\textup{mod}~5^k),
\end{align}which readily gives
\begin{align}
    \label{SecondCongruenceSample2} p_{-(5^{k+m}-3b)/b}\bigg(5^{2(m+1)}\cdot n+5^{2m+1}\cdot r+\dfrac{5^{2(m+1)}-1}{8}\bigg)&\equiv 0 ~(\textup{mod}~5^k),\quad r\in\{1,2,3,4\},
    \end{align}
 and this is the case $s=m+1$ of \eqref{Theorem.5k.Ordinary.Congruence3}. Congruences \eqref{SecondCongruenceSample1} and \eqref{SecondCongruenceSample2} along with the others found in the intermediate steps infer \eqref{Theorem.5k.Ordinary.Congruence3}.

  The following table contains the crucial identities required  in proving the other congruences.
\begin{center}
 \begin{tabular}{c c}
 \hline
 Congruences  & Corresponding identity similar to \eqref{Similar-Identity1}\T\B\\  \hline
 \eqref{Theorem.3k.Ordinary.Congruence2} &  $\big[q^{3 n+1}\big]\{E_1^3\}=-3E_3^3$\T\B\\
  \eqref{Theorem.5k.Ordinary.Congruence4}& $\big[q^{5 n+4}\big]\{E_1^4\}=-5E_5^4$\T\B\\
  \eqref{Theorem.7k.Ordinary.Congruence2}& $\big[q^{7n+6}\big]\{E_1^3\}=-7E_{7}^3$\T\B\\\hline
\end{tabular}
\end{center}
\end{proof}

\begin{remark}\label{mplus1}
 When $m+1\geq k$, it immediately follows from \eqref{ForRemark} that, for all $n\ge 0$,
\begin{align*}
    p_{-(5^{k+m}-3b)/b}\bigg(5^{2m+1}\cdot n+\dfrac{5^{2(m+1)}-1}{8}\bigg)&\equiv0~(\textup{mod}~5^k).
\end{align*}

In a similar way, the following congruences also follow from the proofs of the other congruences in Theorem \ref{3-4powerE1-3-5-7}. For all $n\ge 0$, we have
\begin{align*}
    p_{-(3^{k+m}-3b)/b}\bigg(3^{2m+1}\cdot n+\dfrac{3^{2(m+1)}-1}{8}\bigg)&\equiv0~(\textup{mod}~3^k),\\
     p_{-(5^{k+m}-4b)/b}\bigg(5^{2m+1}\cdot n+\dfrac{5^{2(m+1)}-1}{6}\bigg)&\equiv0~(\textup{mod}~5^k),\\\intertext{and}
      p_{-(7^{k+m}-3b)/b}\bigg(7^{2m+1}\cdot n+\dfrac{7^{2(m+1)}-1}{8}\bigg)&\equiv0~(\textup{mod}~7^k).
\end{align*}

Setting $k=1$, $m=0$, and $b=1$ in the second congruence of the above,
we readily arrive at  Ramanujan's famous congruence
$$p(5n+4)\equiv 0~(\textup{mod}~5).$$
\end{remark}

\begin{proof}[Proof of Theorem \ref{6-8-14powerE1-3-5-7}]
The proofs of all the congruences are similar. Therefore, we prove only congruence \eqref{Theorem.3k.Ordinary.Congruence1}. Choosing $d=6$ and $a=3^k-bd=3^k-6b$ in \eqref{arrangement1}, where $k>1$ such that $(3^k-6b,b)=1$, and then by employing Lemma \ref{UnderModulo}, we have
\begin{align}
   \label{start1} \sum_{n=0}^\infty p_{-(3^k-6b)/b}(n)q^n&=\dfrac{E_1^6}{E_1^{3^k/b}}\equiv\dfrac{E_1^6}{E_3^{3^{k-1}/b}}  ~(\textup{mod}~3^k).
\end{align}
From \eqref{3E_1}, we have
\begin{align}
\label{SimilarIdentity2}\big[q^{3n+2}\big]\left\{E_1^6\right\}=3^2E_3^6.
\end{align}
Applying $\big[q^{3n+2}\big]$ in \eqref{start1}, and then employing the above, we find that
\begin{align}
    \label{end1}\sum_{n=0}^\infty p_{-(3^k-6b)/b}(3n+2)q^n&\equiv3^2\dfrac{E_3^6}{E_1^{3^{k-1}/b}}\notag\\
    &\equiv3^2\dfrac{E_3^6}{E_3^{3^{k-2}/b}}  ~(\textup{mod}~3^k),
\end{align}
which yields
\begin{align}
    \label{CongruenceSample1}p_{-(3^k-6b)/b}(3^2n+3r+2)&\equiv 0 ~(\textup{mod}~3^k), \quad r\in\{1,2\},
\end{align} which is the case $s=1$ of \eqref{Theorem.3k.Ordinary.Congruence1}.

Again, applying  $\big[q^{3n}\big]$ in \eqref{end1}, we obtain
\begin{align*}
     \sum_{n=0}^\infty p_{-(3^k-6b)/b}(3^2n+2)q^n&\equiv3^2\dfrac{E_1^6}{E_1^{3^{k-2}/b}}\\
     &\equiv3^2\dfrac{E_1^6}{E_3^{3^{k-3}/b}}  ~(\textup{mod}~3^k).
\end{align*}
The procedure from \eqref{start1} to \eqref{end1} can be iterated now in the above. In fact, after the $\floor*{k/2}$th iteration, we find that
\begin{align}
    \label{ForOdd}\sum_{n=0}^\infty p_{-(3^k-6b)/b}\bigg(3^{2\floor*{k/2}-1}\cdot n+\dfrac{3^{2\floor*{k/2}}-1}{4}\bigg)q^n&\equiv(3^2)^{\floor*{k/2}}\dfrac{E_3^6}{E_3^{3^{k-2\floor*{k/2}}/b}}~(\textup{mod}~3^k),
\end{align} which gives
\begin{align}
    \label{CongruenceSample2} p_{-(3^k-6b)/b}\bigg(3^{2\floor*{k/2}}\cdot n+3^{2\floor*{k/2}-1}\cdot r+\dfrac{3^{2\floor*{k/2}}-1}{4}\bigg)&\equiv 0 ~(\textup{mod}~3^k),\quad r\in\{1,2\},
\end{align} which is the case $s=\floor*{k/2}$ of \eqref{Theorem.3k.Ordinary.Congruence1}. Congruences \eqref{CongruenceSample1} and \eqref{CongruenceSample2} along with the others found in the intermediate procedures imply  \eqref{Theorem.3k.Ordinary.Congruence1}.

The proofs of \eqref{Theorem.5k.Ordinary.Congruence1} --- \eqref{Theorem.7k.Ordinary.Congruence1} follow in a similar way. The following table contains the salient information for the proofs.
\begin{center}
 \begin{tabular}{c c}
 \hline
 Congruences  & Corresponding identity similar to \eqref{SimilarIdentity2}\T\B\\  \hline
 \eqref{Theorem.5k.Ordinary.Congruence1} &  $\big[q^{5n+3}\big]\{E_1^8\}=-125E_5^8$\T\B\\
  \eqref{Theorem.5k.Ordinary.Congruence2}& $\big[q^{5 n+4}\big]\{E_1^{14}\}=-15625E_5^{14}$\T\B\\
  \eqref{Theorem.7k.Ordinary.Congruence1}& $\big[q^{7n+5}\big]\{E_1^6\}=49qE_{7}^{6}$\T\B\\\hline
\end{tabular}
\end{center}
\end{proof}

\begin{remark} \label{remark11} We note that when $k$ is even, no further iteration can be executed because $k-2\floor*{k/2}$ vanishes. We also observe that \eqref{ForOdd} gives more information than what we actually have in \eqref{Theorem.3k.Ordinary.Congruence1}.  In fact, from \eqref{ForOdd}, we have
\begin{align*}
      p_{-(3^k-6b)/b}\bigg(3^{k-1}\cdot n+\dfrac{3^{k}-1}{4}\bigg)&\equiv 0~(\textup{mod}~3^k),
\end{align*}
which actually contains the congruence \eqref{CongruenceSample2}.

When $k$ is odd, one more iteration can be executed and that leads to Theorem \ref{B-6-8-14-powerE1-3-5-7}.

Note that this remark also applies to all other congruences in Theorem \ref{6-8-14powerE1-3-5-7}.
\end{remark}

\begin{proof}[Proof of Theorem \ref{B-6-8-14-powerE1-3-5-7}]
The proofs of all the congruences are similar. Therefore, we prove only congruence \eqref{B.Theorem.3k.Ordinary.Congruence1} and omit the proofs of the remaining congruences. When $k$ is odd,    $k-2\floor*{k/2}=1$. Therefore, as mentioned in Remark \ref{remark11}, the procedure of extraction may be applied once more in \eqref{ForOdd}. To that end, applying $\big[q^{3n}\big]$ in \eqref{ForOdd}, we find that
\begin{align*}
    \sum_{n=0}^\infty p_{-(3^k-6b)/b}\bigg(3^{2\floor*{k/2}}\cdot n+\dfrac{3^{2\floor*{k/2}}-1}{4}\bigg)q^n&\equiv(3^2)^{\floor*{k/2}}\dfrac{E_1^6}{E_1^{3/b}}\\
    &\equiv(3^2)^{\floor*{k/2}}\dfrac{E_1^6}{E_3^{1/b}}~(\textup{mod}~3^k).
\end{align*}
Applying $\big[q^{3n+2}\big]$ in the above, we obtain
\begin{align*}
     \sum_{n=0}^\infty p_{-(3^k-6b)/b}\bigg(3^{2\floor*{k/2}+1}\cdot n+\dfrac{ 3^{2\floor*{k/2}+1+1}-1}{4}\bigg)q^n&\equiv(3^2)^{\floor*{k/2}+1}\dfrac{E_3^6}{E_1^{1/b}}~(\textup{mod}~3^k),
\end{align*}
which gives the $3^k$-balanced congruence \eqref{B.Theorem.3k.Ordinary.Congruence1}.
\end{proof}

\section{\bf Proofs of Theorems \ref{Theorem.Color12.1} -- \ref{Theorem.Color14.2}} \label{sec5}

\begin{proof}[Proof of Theorem \ref{Theorem.Color12.1}]
Since $\ell\mid(a+db)$, we assume that $a+db=\ell m$ for some integer $m$. Furthermore, since gcd$(a,b)=1$, it follows that  $(b,\ell)=1$. From \eqref{arrangement2},  we have
\begin{align}
     \label{arrangement1undermodulo}\sum_{n=0}^\infty  p_{[1,2;-a/b]}(n)q^n&=\dfrac{(E_1E_2)^d}{(E_1 E_{2})^{\ell m/b}}\equiv \dfrac{(E_1E_2)^d}{(E_\ell E_{2\ell})^{m/b}}~(\textup{mod}~\ell).
\end{align}

We now proceed with the following two cases for $d$.

\noindent \textit{Case} $d=2$. By \eqref{JacobiTheorem} and \eqref{PsiFunction},  we have
\begin{align}\label{E1E2-Th1.10d2}
    (E_1E_2)^2&=E_1^3\cdot \dfrac{E_2^2}{E_1}=\sum_{n=-\infty}^{\infty}\sum_{m=-\infty}^{\infty}(4n+1)q^{((4n+1)^2+(4m+1)^2-2)/8}.
\end{align}
Note that
\begin{align*}
    N&=((4n+1)^2+(4m+1)^2-2)/8
\end{align*}is equivalent to
\begin{align*}
    8N+2&=(4n+1)^2+(4m+1)^2.
\end{align*}
If $\ell\equiv3~(\textup{mod}~4)$, then ${\Leg{-1 }{\ell}}=-1$. Therefore,
\begin{align*}
    8N+2&\equiv0~(\textup{mod}~\ell),\\\intertext{equivalently}
    4N+1&\equiv0~(\textup{mod}~\ell),
\end{align*}  if and only if
\begin{alignat*}{2}
   &4n+1\equiv0~(\textup{mod}~\ell)\hspace{.5cm}\textup{and}\hspace{.5cm}&&4m+1\equiv0~(\textup{mod}~\ell).
\end{alignat*}
Using \eqref{E1E2-Th1.10d2} in \eqref{arrangement1undermodulo} and then comparing the coefficients of $q^{\ell n+r}$ on both sides, we arrive at  \eqref{General.Congruence.1} for the case $d=2$.

\noindent \textit{Case} $d=3$. By \eqref{JacobiTheorem}, we have
\begin{align}\label{E1E2-Th1.10d3}
    (E_1E_2)^3&=\sum_{n=-\infty}^{\infty}\sum_{m=-\infty}^{\infty}(4n+1)(4m+1)q^{((4n+1)^2+2(4m+1)^2-3)/8}.
\end{align}
Note that
\begin{align*}
    N&=((4n+1)^2+2(4m+1)^2-3)/8
\end{align*}is equivalent to
\begin{align*}
    8N+3&=(4n+1)^2+2(4m+1)^2.
\end{align*}
If $\ell\equiv5~\textup{or}~7~(\textup{mod}~8)$, then ${\Leg{-2}{\ell}}=-1$. From the above, it follows that
\begin{align*}
    8N+3&\equiv0~(\textup{mod}~\ell)
\end{align*}if and only if
\begin{alignat*}{2}
   &4n+1\equiv0~(\textup{mod}~\ell)\hspace{.5cm}\textup{and}\hspace{.5cm}&&4m+1\equiv0~(\textup{mod}~\ell).
\end{alignat*}
Employing \eqref{E1E2-Th1.10d3} in \eqref{arrangement1undermodulo} and then comparing the coefficients of $q^{\ell n+r}$ on both sides, we obtain  \eqref{General.Congruence.1} for this case.
\end{proof}

\begin{proof}[Proof of Theorem \ref{Theorem.Color13.1}] For  some integer $m$, we have $a+db=\ell m$  as $\ell\mid(a+db)$. Furthermore, gcd$(a,b)=1$ implies that $(b,\ell)=1$. From \eqref{arrangement2},  we have
\begin{align}
     \label{arrangement1-3undermodulo}\sum_{n=0}^\infty  p_{[1,3;-a/b]}(n)q^n&=\dfrac{(E_1E_3)^d}{(E_1 E_{3})^{\ell m/b}}\equiv \dfrac{(E_1E_3)^d}{(E_\ell E_{3\ell})^{m/b}}~(\textup{mod}~\ell).
\end{align}
By \eqref{JacobiTheorem}, we have
\begin{align}\label{E1E3-Th1.11d3}
    (E_1E_3)^3&=\sum_{n=-\infty}^{\infty}\sum_{m=-\infty}^{\infty}(4n+1)(4m+1)q^{((4n+1)^2+3(4m+1)^2-4)/8}.
\end{align}
Note that
\begin{align*}
    N&=((4n+1)^2+3(4m+1)^2-4)/8
\end{align*}is equivalent to
\begin{align*}
    8N+4&=(4n+1)^2+3(4m+1)^2.
\end{align*}
If $\ell\equiv5~\textup{or}~11~(\textup{mod}~12)$, then ${\Leg{-3}{\ell}}=-1$. It follows that
\begin{align*}
    8N+4&\equiv0~(\textup{mod}~\ell),\\\intertext{equivalently,} 2N+1&\equiv0~(\textup{mod}~\ell)
\end{align*}if and only if
\begin{alignat*}{2}
   &4n+1\equiv0~(\textup{mod}~\ell)\hspace{.5cm}\textup{and}\hspace{.5cm}&&4m+1\equiv0~(\textup{mod}~\ell).
\end{alignat*}
Employing \eqref{E1E3-Th1.11d3} in \eqref{arrangement1-3undermodulo} and then comparing the coefficients of $q^{\ell n+r}$ on both sides, we arrive at \eqref{General.Congruence.13.1} to complete the proof.
\end{proof}

\begin{proof}[Proof of Theorem \ref{Theorem.Color14.1}] It follows from gcd$(a,b)=1$ and  $\ell\mid(a+db)$ that $(b,\ell)=1$. We assume that
$a+db=\ell m$ for some integer $m$. From \eqref{arrangement2}, we find that
\begin{align}
     \label{arrangement1undermodulo14}\sum_{n=0}^\infty  p_{[1,4;-a/b]}(n)q^n&=\dfrac{(E_1E_4)^d}{(E_1 E_{4})^{\ell m/b}}~\equiv\dfrac{(E_1E_4)^d}{(E_\ell E_{4\ell})^{m/b}}~(\textup{mod}~\ell).
\end{align}

We now proceed according to the two values of $d$.

\noindent \textit{Case} $d=2$. By \eqref{PentagonalTheorem} and \eqref{FoundInHirschhorn},  we have
\begin{align}\label{E1E4-Th1.11d2}
    (E_1E_4)^2&=E_2\cdot \dfrac{E_1^2E_4^2}{E_2}=\sum_{n=-\infty}^{\infty}\sum_{m=-\infty}^{\infty}(-1)^n(3m+1)q^{((6n+1)^2+4(3m+1)^2-5)/12}.
\end{align}
Clearly,
\begin{align*}
    N&=((6n+1)^2+4(3m+1)^2-5)/12
\end{align*}is equivalent to
\begin{align*}
    12N+5&=(6n+1)^2+4(3m+1)^2.
\end{align*}
If $\ell\equiv3~(\textup{mod}~4)$, then ${\Leg{-4 }{\ell}}=-1$. It follows that
\begin{align*}
    12N+5&\equiv0~(\textup{mod}~\ell)
\end{align*}if and only if
\begin{alignat*}{2}
   &6n+1\equiv0~(\textup{mod}~\ell)\hspace{.5cm}\textup{and}\hspace{.5cm}&&3m+1\equiv0~(\textup{mod}~\ell).
\end{alignat*}
Using \eqref{E1E4-Th1.11d2} in \eqref{arrangement1undermodulo14} and then comparing the coefficients of $q^{\ell n+r}$ on both sides, we readily arrive at  \eqref{General.Congruence.2} for the case $d=2$.

\noindent \textit{Case} $d=3$. By \eqref{JacobiTheorem}, we have
\begin{align}\label{E1E4-Th1.11d3}
    (E_1E_4)^3&=\sum_{n=-\infty}^{\infty}\sum_{m=-\infty}^{\infty}(4n+1)(4m+1)q^{((4n+1)^2+4(4m+1)^2-5)/8}.
\end{align}
We note that
\begin{align*}
    N&=((4n+1)^2+4(4m+1)^2-5)/8
\end{align*}is equivalent to
\begin{align*}
    8N+5&=(4n+1)^2+4(4m+1)^2.
\end{align*}
If $\ell\equiv3(\textup{mod}~4)$, then ${\Leg{-4}{\ell}}=-1$. It follows that
\begin{align*}
    8N+5&\equiv0~(\textup{mod}~\ell)
\end{align*}if and only if
\begin{alignat*}{2}
   &4n+1\equiv0~(\textup{mod}~\ell)\hspace{.5cm}\textup{and}\hspace{.5cm}&&4m+1\equiv0~(\textup{mod}~\ell).
\end{alignat*}
Congruence \eqref{General.Congruence.2} now follows by employing \eqref{E1E4-Th1.11d3} in \eqref{arrangement1undermodulo14} and then comparing the coefficients of  $q^{\ell n+r}$ on both sides.
\end{proof}

\begin{proof}[Proof of Theorem \ref{Theorem.Color14.2}]
Setting $r=4$, $d=3$ and $a=5^k-bd=5^k-3b$ in \eqref{arrangement2}, we have
\begin{align*}
    \sum_{n=0}^\infty p_{[1,4;-(5^k-3b)/b]}(n)q^n=\dfrac{(E_{1}E_{4})^3}{(E_{1}E_{4})^{5^{k}/b}}\equiv \dfrac{(E_{1}E_{4})^3}{(E_{5}E_{20})^{5^{k-1}/b}}~(\textup{mod}~5).
\end{align*}
Employing the $5$-dissections of $E_1^3$ and $E_4^3$ arising from \eqref{E1R} in the above, and then applying $\big[q^{5n+2}\big]$ and $\big[q^{5n+3}\big]$, we obtain
\begin{align*}
    \sum_{n=0}^\infty p_{[1,4;-(5^k-3b)/b]}(5n+2)q^n&\equiv 5\dfrac{(E_{5}E_{20})^3}{(E_{1}E_{4})^{5^{k-1}/b}}\bigg(-3q\bigg(R^2(q^4)+\dfrac{q^2}{R^2(q)}\bigg)\notag\\
    &\quad+q^2\bigg(R^3(q)-\dfrac{q^3}{R^3(q^4)}\bigg)\bigg)~(\textup{mod}~5),\\\intertext{and}
    \sum_{n=0}^\infty p_{[1,4;-(5^k-3b)/b]}(5n+3)q^n&\equiv 5\dfrac{(E_{5}E_{20})^3}{(E_{1}E_{4})^{5^{k-1}/b}}\bigg(\bigg(R^3(q^4)-\dfrac{q^3}{R^3(q)}\bigg)\notag\\
    &\quad-3q^2\bigg(R^2(q)-\dfrac{q^2}{R^2(q^4)}\bigg)\bigg)~(\textup{mod}~5),
\end{align*}
respectively, which readily implies \eqref{Theorem.5k.Color14.2.Congruence}.
\end{proof}

\section{\bf Proofs of Theorems \ref{Theorem.(m+1)/2} -- \ref{balancedtheorem}}\label{sec6}
\begin{proof}[Proof of Theorem \ref{Theorem.(m+1)/2}] The proofs of the congruences are similar in nature. So we present a detailed proof only for \eqref{Theorem.3k.Color12.Congruence3}. For the proofs of the remaining congruences we mention only the crucial stages.

Setting $r=2$, $d=1$ and $a=3^{k}-db$, for integer $k> 1$ in \eqref{arrangement2} and for integer $s\ge1,$ we have
\begin{align}
     \label{Th1.13-3k-a}\sum_{n=0}^\infty p_{[1,2;-(3^{k}-b)/b]}(n)q^n&=\dfrac{E_1E_2}{(E_1E_2)^{3^{k}/b}}\equiv\dfrac{E_1E_2}{(E_3E_6)^{3^{k-1}/b}}~(\textup{mod}~3^{k-2s+1}).
\end{align}

Now, from \eqref{3E_1E_2}, we note that
\begin{align}\label{Th1.13-3k-recur}
   \left[q^{3n+1}\right]\{E_1E_2\}=-E_3E_6.
   \end{align}
Applying $\left[q^{3n+1}\right]$ in \eqref{Th1.13-3k-a} and then using the above, we find that
\begin{align*}
    \sum_{n=0}^\infty p_{[1,2;-(3^{k}-b)/b]}(3n+1)q^n&\equiv-\dfrac{E_3E_6}{(E_1E_2)^{3^{k-1}/b}}~(\textup{mod}~3^{k-2s+1}),
\end{align*}
which is equivalent to
\begin{align}\label{Th1.13-3k-c}
     \sum_{n=0}^\infty p_{[1,2;-(3^{k}-b)/b]}(3n+1)q^n&\equiv-\dfrac{E_3E_6}{(E_3E_6)^{3^{k-2}/b}}~(\textup{mod}~3^{k-2s+1}).
\end{align}
The above readily implies
\begin{align*}
     p_{[1,2;-(3^{k}-b)/b]}(3^2n+3r+1)q^n&\equiv0~(\textup{mod}~3^{k-2s+1}), \quad r\in\{1,2\},
\end{align*}
and when $s = 1,$ we obtain the highest possible modulus for the above congruence which is
the case  $s=1$ of \eqref{Theorem.3k.Color12.Congruence3}.

 Next, applying $\left[q^{3n}\right]$ in \eqref{Th1.13-3k-c}, we have
 \begin{align*}
    \sum_{n=0}^\infty p_{[1,2;-(3^{k}-b)/b]}(3^2n+1)q^n&\equiv-\dfrac{E_1E_2}{(E_1E_2)^{3^{k-2}/b}}~(\textup{mod}~3^{k-2s+1}).
\end{align*}
Now, to move ahead, we take $s=2$ in the above. Once again employing \eqref{3E_1E_2} in the above, and then applying $\left[q^{3n+1}\right]$, we find that
\begin{align*}
     \sum_{n=0}^\infty p_{[1,2;-(3^{k}-b)/b]}(3^3n+3^2+1)q^n&\equiv\dfrac{E_3E_6}{(E_1E_2)^{3^{k-3}/b}}\\
     &\equiv\dfrac{E_3E_6}{(E_3E_6)^{3^{k-4}/b}}~(\textup{mod}~3^{k-2s+1}),
\end{align*}
which implies
\begin{align*}
     p_{[1,2;-(3^{k}-b)/b]}(3^4n+3^3r+3^2+1)q^n&\equiv0~(\textup{mod}~3^{k-2s+1}),\quad  r\in\{1,2\},
\end{align*}
and when $s = 2,$ we obtain the highest possible modulus for the above congruence which is
the case  $s=2$ of \eqref{Theorem.3k.Color12.Congruence3}. We complete the proof by noting that the process can be iterated till the last case when $s=\floor{k/2}$ to arrive at \eqref{Theorem.3k.Color12.Congruence3}.

The proofs of the congruences \eqref{Theorem.5k.Color12.Congruence4} -- \eqref{Theorem.11k.Color14.Congruence1} can be accomplished in a similar way. We notice that \eqref{Th1.13-3k-recur} was crucial in proving \eqref{Theorem.3k.Color12.Congruence3}. Therefore, in the following table we present only the corresponding identities similar to \eqref{Th1.13-3k-recur}.
\begin{center}
 \begin{tabular}{c c}
 \hline
 Congruences  & Corresponding identity similar to \eqref{Th1.13-3k-recur}\T\B\\  \hline
 \eqref{Theorem.5k.Color12.Congruence4} &  $\left[q^{5n+3}\right]\{E_1E_2\}=E_5E_{10}$\T\B\\
  \eqref{Theorem.5k.Color13.Congruence1Duplicate}& $\left[q^{5n+4}\right]\{E_1E_3\}=E_5E_{15}$\T\B\\
  \eqref{Theorem.7k.Color12.Congruence1}& $\left[q^{7n+6}\right]\{E_1E_2\}=E_7E_{14}$\T\B\\
  \eqref{Theorem.7k.Color14.Congruence1}& $\left[q^{7n+3}\right]\{E_1E_4\}=qE_7E_{28}$\T\B\\
  \eqref{Theorem.11k.Color13.Congruence1}& $\left[q^{11n+9}\right]\{E_1E_3\}=qE_{11}E_{33}$\T\B\\
  \eqref{Theorem.11k.Color14.Congruence1}& $\left[q^{11n+3}\right]\{E_1E_4\}=q^2E_{11}E_{44}$\T\B\\\hline
\end{tabular}
\end{center}
\end{proof}
\begin{proof}[Proof of Theorem \ref{Theorem.m}] Setting $r=2$, $d=2$ and $a=3^{k+m}-db$ for positive integers $m$ and $k\ge 1$, in \eqref{arrangement2} and for integer $s\ge1$ we find that
\begin{align*}
     \sum_{n=0}^\infty p_{[1,2;-(3^{k+m}-2b)/b]}(n)q^n&=\dfrac{(E_1E_2)^2}{(E_1E_2)^{3^{k+m}/b}}\equiv\dfrac{(E_1E_2)^2}{(E_3E_6)^{3^{k-1+m}/b}}~(\textup{mod}~3^{k+m-s+1}).
\end{align*}

It can be easily seen from \eqref{3E_1E_2} that
\begin{align*}
   \left[q^{3n+2}\right]\{(E_1E_2)^2\}=-3E_3^2E_6^2.
   \end{align*}
  The remaining part of the proof of  \eqref{Theorem.3k.Color12.Congruence2} is similar to that of \eqref{Theorem.3k.Color12.Congruence3} given above.

  Next, to prove \eqref{Theorem.5k.Color13.Congruence2}, we set $r=3$, $d=2$ and $a=5^{k+m}-db$ for integers $m,$ and $k\ge 1$, in \eqref{arrangement2} and for integer $s\ge1$, we have
\begin{align*}
     \label{Th1.14-5k-a}\sum_{n=0}^\infty p_{[1,3;-(5^{k+m}-2b)/b]}(n)q^n&=\dfrac{(E_1E_3)^2}{(E_1E_3)^{5^{k+m}/b}}
     \equiv\dfrac{(E_1E_3)^2}{(E_5E_{15})^{5^{k-1+m}/b}}~(\textup{mod}~5^{k+m-s+1}).
\end{align*}

Now, employing \eqref{E1R}, we have
\begin{align*}
   \left[q^{5n+3}\right]\{(E_1E_3)^2\}=E_5^2E_{15}^2\left(q-2\left(R^2(q) R(q^3)-\dfrac{R^2(q^3)}{R(q)}+q^2\dfrac{R(q)}{R^2(q^3)}-\dfrac{q^2}{R^2(q) R(q^3)}\right)\right),
   \end{align*}
 which, by \eqref{Identity2}, gives
\begin{align*}
   \left[q^{5n+3}\right]\{(E_1E_3)^2\}=-5qE_5^2E_{15}^2.
   \end{align*}
   Now the rest of the proof is similar to that of \eqref{Theorem.3k.Color12.Congruence3}. So we omit the details.
\end{proof}

\begin{remark}As in Remark \ref{mplus1}, here also we note that when $m+1\geq k$, then, for all $n\ge 0$,
\begin{align*}
     p_{[1,2;-(3^{k+m}-2b)/b]}\bigg(3^{2m+1}\cdot n+\dfrac{3^{2(m+1)}-1}{4}\bigg)&\equiv0~(\textup{mod}~3^k)\\\intertext{and}
     p_{[1,3;-(5^{k+m}-2b)/b]}\bigg(5^{2m+1}\cdot n+\dfrac{2\cdot5^{2m+1}-1}{3}\bigg)&\equiv0~(\textup{mod}~5^k).
\end{align*}

Setting $k=1$, $m=0$, and $b=1$ in the last congruence, we have
$$p_{[1,3;-3]}(5n+3)\equiv 0~(\textup{mod}~5),$$
which seems to be new.
\end{remark}

\begin{proof}[Proof of Theorem \ref{Theorem.k/2}] To prove \eqref{Theorem.3k.Color12.Congruence1}, \eqref{Theorem.5k.Color12.Congruence1} and \eqref{Theorem.5k.Color13.Congruence1} we use, respectively, $3$-dissection of $(E_1E_2)^5$ that arises from \eqref{3E_1E_2} and $5$-dissections of $(E_1E_2)^3$ and $(E_1E_3)^3$ that arise from \eqref{E1R}. The proofs can be accomplished by using some identities similar to \eqref{Th1.13-3k-recur}. The following table contains the corresponding identities.\begin{center}
 \begin{tabular}{c c}
 \hline
 Congruences  & Corresponding identity similar to \eqref{Th1.13-3k-recur}\T\B\\  \hline
 \eqref{Theorem.3k.Color12.Congruence1} &  $\left[q^{3n+2}\right]\{(E_1E_2)^5\}=-81qE_3^5E_{6}^5$\T\B\\
  \eqref{Theorem.5k.Color12.Congruence1}& $\left[q^{5n+4}\right]\{(E_1E_2)^3\}=25qE_5^3E_{10}^3$\T\B\\
  \eqref{Theorem.5k.Color13.Congruence1}& $\left[q^{5n+2}\right]\{(E_1E_3)^3\}=25q^2E_5^3E_{15}^3$\T\B\\\hline
\end{tabular}
\end{center}
\end{proof}
\begin{proof}[Proof of Theorem \ref{balancedtheorem}] The balanced congruences \eqref{B.Theorem.3k.Color12.Congruence1} -- \eqref{B.Theorem.5k.Color13.Congruence1} follow by extending the proofs of the congruences in Theorem \ref{Theorem.k/2} to odd $k$. Since the process  has already been discussed while proving Theorem \ref{B-6-8-14-powerE1-3-5-7}, we omit the details.
\end{proof}

\section*{\bf Acknowledgement}
 The first author was partially supported by Grant no. MTR/2018/000157 of Science \& Engineering Research Board (SERB), DST, Government of India under the MATRICS scheme. The second author was partially supported by Council of Scientific \& Industrial Research (CSIR), Government of India under CSIR-JRF scheme. The authors thank both the funding agencies.

\end{document}